\documentclass[12pt]{amsart}
\usepackage[utf8]{inputenc}
\usepackage[english]{babel}
\usepackage{enumerate}
\usepackage{amsthm,amsfonts,amssymb,amsmath}

\newtheorem{theorem}{Theorem}[section]
\newtheorem{proposition}[theorem]{Proposition}
\newtheorem{lemma}[theorem]{Lemma}
\newtheorem{corollary}[theorem]{Corollary}
\theoremstyle{definition}
\newtheorem{definition}[theorem]{Definition}

\title{Characterizations of $G$-ANR spaces and inverse limits}

\author[Antonyan]{Sergey A. Antonyan}
\address{Departamento de Matem\'aticas, Facultad de Ciencias, 
Universidad Nacional Aut\'onoma de M\'exico, 04510, Ciudad de M\'exico, M\'exico}
\email{antonyan@unam.mx}

\author[Kant\'un-Montiel]{Aura Lucina Kant\'un-Montiel}
\address{Instituto de Agroingenier\'ia, Universidad del Papaloapan, Av. Ferrocarril S/{N}, Ciudad  Universitaria, 
CP 68400, Loma Bonita, Oaxaca, M\'exico.}
\email{alkantun@unpa.edu.mx}

\author[Mata-Cano]{Jes\'us Eduardo Mata-Cano}
\address{Facultad de Ciencias Exactas, Universidad Ju\'arez del Estado de Durango, Lic. H\'ector Garc\'ia Calder\'on 210, CP 34113, SARH, Durango,
 M\'exico.}
 \email{eduardo.mata@ujed.mx}

\author[Mata-Romero]{Armando Mata-Romero}
\address{Facultad de Ciencias Exactas, Universidad Ju\'arez del Estado de Durango, Lic. H\'ector Garc\'ia Calder\'on 210, CP 34113, SARH, Durango M\'exico.}
\email{armandomr@ujed.mx}

\thanks{This paper was partially supported by grant              IN-107426 of PAPIIT-DGAPA (UNAM)}

\begin{document}

\begin{abstract}
    In this paper we prove that, for a compact group $G$, a metrizable $G$-space is a $G$-ANR under the following asumptions: 
(1) if it dominates a $G$-ANR space through a fine $G$-homotopy equivalence;
(2) if it is $G$-homotopy dense in a $G$-ANR;
(3) if it contains a $G$-ANR as a $G$-homotopy dense subset;
(4) if it is the inverse limit of an inverse sequence of $G$-ANR spaces with bonding maps that are fine $G$-homotopy equivalences.
\end{abstract}

\maketitle

\section{Introduction}

This paper examines the equivariant topology of metrizable $G$-spaces  
(i.e., metrizable spaces equipped with a continuous action of a compact topological group $G$) from the perspective of the equivariant theory of retracts. 

We introduce and analyze equivariant versions of such concepts  as homotopy dense subsets (see \cite{TRZ} and \cite{Taras}), fine homotopy equivalences and $h$-refinements. 
Our approach is build on classical results in the theory of retracts, extending them to equivariant settings and providing new insights into the behavior of $G$-ANR spaces under fine $G$-homotopy equivalences.

The main results of this paper establish sufficient conditions under which metrizable $G$-spaces inherit the $G$-ANR property. 
Theorem \ref{Gkozlowski} proves that if a metrizable $G$-space dominates a $G$-ANR space through a fine $G$-homotopy equivalence then it is itself a $G$-ANR. 
Furthermore, Theorem \ref{gdenso-homofina} characterizes  $G$-homotopy dense subsets, showing their equivalence to subsets whose inclusion map is a fine $G$-homotopy equivalence. 
Another significant contribution  is Theorem \ref{GCurtis}, which  extends Curtis’s theorem \cite{Curtis71} to the equivariant setting;  
this result establishes that  the  limit of an inverse sequence of completely metrizable $G$-ANR spaces, where the bonding maps are fine $G$-homotopy equivalences, is itself a $G$-ANR.  
Additionally, if the spaces in the inverse sequence are $G$-ARs, then the inverse limit also inherits the $G$-AR property. 
These findings not only generalize classical theorems, but also provide a framework for understanding the structure of inverse limits in the category of $G$-spaces.

The fundamental concepts and results concerning the theory of $G$-spaces are drawn from \cite{Bredon72} and \cite{Palais60}. For the equivariant theory of retracts the principal references include \cite{Antonyan1987},  \cite{Antonyan1988}, \cite{Antonyan2001}  and  \cite{Antonyan2005}. 

\medskip

\section{Basic notations}

Throughout this paper, we  denote by $G$ a compact Hausdorff topological group, unless otherwise specified. The identity element of $G$ will be denoted by $e$.
All topological spaces are assumed to be completely regular and Hausdorff. 
All maps are assumed to be continuous.

An action of $G$ on a space $X$ is a map $G\times X\to X$, $(g,x)\mapsto g x$,  such that $ex=x$ for all $x\in X$, and $h(gx)=(hg)x$ for all $g,h\in G$, $x\in X$. By a $G$-{\it space} we mean a topological space $X$ equipped with a  continuous action of $G$.

Let $X$ be a $G$-space and let $x \in X$. 
The $G$-space $G(x)= \{gx \ | \ g\in G \}$ is called the {\it G-orbit} of $x$. 
The set of all the $G$-orbits of $X$  is denoted by $X/G$.    
A subset $S \subset X$,  is called {\it $G$-invariant} if $ S = G( S) = \{gs\mid g\in G, s\in S\}$.

Let  $X$ and $Y$ be $G$-spaces. A continuous map $f:X\rightarrow Y$
  is called a $G$-{\it map} or an {\it equivariant map}, if  $f(gx) = gf(x)$
  for every $(g,x) \in G\times X$. 
  If $G$ acts trivially on $Y$, we refer to $f$ as an {\it invariant map}.

A homotopy $F:X\times I\rightarrow Y $, where $I = [0,1]$,
  is called a $G$-{\it homotopy} if it is a $G$-map with  $X\times I$ 
 carrying the diagonal action $g(x,t) = (gx,t)$.
 For each $t\in I$, we  denote by $F_t$ the induced $G$-map 
$F_t:{X}\rightarrow  {Y}$ given by $F_t(x)=F(x,t)$.
 Two $G$-maps $f_{0},f_{1}:X\to Y$ are $G$-homotopic if there exists a $G$-homotopy $F:X\times I\to Y$ such that $F_0=f_{0}$ and $F_1=f_{1}$. In this case, we write $f_{0}\underset{G}{\simeq}   f_{1}$. 

\medskip

Let $G$-$\mathcal{M}$ denote the class of all metrizable $G$-spaces. It is well known that every space $X\in G$-$\mathcal{M}$ admits a compatible invariant metric, that is, a  metric $\rho:X\times X\to \mathbb{R}$ such that $\rho(gx,gy)=\rho(x,y)$ for all $g\in G$ and $x,y\in X$ (see \cite[Proposition 1.1.12]{Palais60}).

\smallskip

A  metrizable $G$-space  $Y$ is called a $G$-equivariant absolute neighborhood retract (for
the class $G$-$\mathcal M$), provided that  for  any closed $G$-embedding
$Y\hookrightarrow X$ in a    metrizable $G$-space $X$, there exists a $G$-retraction $r\colon
U\to Y$, where $U$ is an invariant neighborhood  of $Y$ in $X$ (notation: $Y\in G$-ANR). If, in addition, one can always take
$U=X$, then we say that $Y$ is a $G$-equivariant absolute retract  (notation: $Y\in G$-AR).

A  $G$-space  $Y$   is called a $G$-equivariant absolute   neighborhood extensor  for the class $G$-$\mathcal M$ (notation: $Y\in G$-ANE)  if,  for any closed invariant subset $A$ of a metrizable $G$-space $X$ and  any  $G$-map $f:A\to Y$,  there exist an invariant neighborhood $U$ of $A$ in $X$ and  a $G$-map $\psi\colon U\to Y$  that  extends $f$. If, in addition, one  can always take $U=X$, then we say that $Y$ is a $G$-equivariant  absolute  extensor  for $G$-$\mathcal M$ (notation: $Y\in G$-AE). The map $\psi$ is called a $G$-extension of $f$.

We note (see \cite{Antonyan1987}) that a metrizable $G$-space is a $G$-ANR (respectively, a $G$-AR) if and only if
it is a  $G$-ANE  (respectively, a $G$-AE).

\medskip

\section{$G$-$\mathcal{U}$-homotopies}

Our interest lies in homotopic properties controlled by a certain degree of proximity. 
In what follows, we present the concepts and results concerning the relationship between $G$-ANRs and $G$-homotopies controlled by open covers.

\medskip
Let $\mathcal{U}$ and $\mathcal{V}$ be two open coverings of a space $X$.
We say that $\mathcal{U}$ is a refinement of $\mathcal{V}$ if for each $U\in \mathcal{U}$ there is some $V\in \mathcal{V}$ such that $U\subset V$.

For a subset $A\subset X$ we denote the {\it star} of $A$ with respect to $\mathcal{U}$ by
$$\text{St}(A,\mathcal{U})=\bigcup\left\{U\in \mathcal{U} \mid \ U\cap A\neq \emptyset\right\}.$$
An open cover $\mathcal{U}$ of $X$ is a star-refinement of  $\mathcal{V}$ if 
$$\text{St}(\mathcal{U})=\{\text{St}(U,\mathcal{U})\mid U\in \mathcal{U}\}$$ forms a refinement of $\mathcal{V}$.

Observe that for any two open covers $\mathcal{U}$ and $\mathcal{V}$ of a $G$-space $X$ there always exists a common refinement, which can be obtained by taking the cover $\mathcal{W}=\{U\cap V\mid U\in\mathcal{U}, \; V\in\mathcal{V}\}$.

\medskip 

Let $X$ and $Y$ be topological spaces and $\mathcal{U}$ an open cover of $Y$.
Two maps $f,g:X\to Y$ are said to be \textit{$\mathcal{U}$-close} if, for every $x\in X$, there exists $U\in \mathcal{U}$ such that $f(x),g(x)\in U$.

  \begin{lemma}\label{remmark:Kozlowski-1}
        Let $f:X\to Y$ be a $G$-map such that $f(X)$ is dense in $Y$.
        Let $\mathcal{U}$ be an open cover of $Y$ and $\varphi: Y\to X$ be a $G$-map such that $\varphi f$ is $f^{-1}(\mathcal{U})$-close to the identity map $id_{X}$.
        Then $f\varphi$ is $\text{St}(\mathcal{U})$-close to $id_{Y}$.
    \end{lemma}
    \begin{proof}
        Let $y\in Y$. Let $U_{1}, U_{2}\in \mathcal{U}$ such that $y\in U_{1}$ and $f\varphi(y)\in U_{2}$. Then  $y\in U_{1}\cap (f\varphi)^{-1}(U_{2})$.  
        Since  $f(X)$ is dense in $Y$, there exists $x\in X$ such that $f(x)\in U_{1}\cap (f\varphi)^{-1}(U_{2})$, implying that $f\varphi f(x)\in U_2$.          
        Moreover, since  $\varphi f$ is $f^{-1}(\mathcal{U})$-close to $id_{X}$, there exists $U_{3}\in \mathcal{U}$ such that $x,\varphi f(x)\in f^{-1}(U_{3})$, i.e., $f(x),f\varphi f(x)\in U_{3}$. 
        Thus, $U_{1}\cap U_{3}\neq \emptyset$ and $U_{2}\cap U_{3}\neq \emptyset$. 
        Therefore $y,f\varphi(y)\in \text{St}(U_{3},\mathcal{U})$, as required.
    \end{proof}

\medskip 
We say that an open cover $\mathcal{U}$ of a $G$-space $X$ is a {\it $G$-cover} if  $gU \in \mathcal{U}$ for every $U \in \mathcal{U}$ and $g \in G$.

\medskip

In  paracompact $G$-spaces, there are arbitrary small $G$-covers, namely, we have the following proposition.

 \begin{proposition}[{\cite[Lemma 4.5]{Antonyan2005}}]\label{gcubierta}
Let $X$ be a paracompact $G$-space. Then for every open cover $\mathcal{U}$ of $X$, there exists a $G$-cover $\mathcal{V}$  such that  $\mathcal{V}$ is a star-refinement of $\mathcal{U}$.
 \end{proposition}

Let $\mathcal{U}$ be an open cover of a $G$-space $Y$. 
A $G$-homotopy $F:X\times I\to Y$ is said to be {\it limited by}  $ \mathcal{U} $, or simply a {\it $G$-$\mathcal{U}$-homotopy}, if for every $x\in X$, 
there exists $U\in \mathcal{U}$ such that $F_t(x)\in U$ for all $t\in I$.
In this  case,  we say that $F_0$ and $F_1$ are $G$-$\mathcal U$-homotopic maps, 
and we write $F_0\underset{G\text{-}\mathcal{U}}{\simeq} F_1$. 

It is clear that, if $\mathcal{V}$ is an open cover of a $G$-space $Y$ such that $\mathcal{U}$ is a refinement of $\mathcal{V}$, then every  $G$-$\mathcal{U}$-homotopy $F:X\times I\to Y$ is also a $G$-$\mathcal{V}$-homotopy.

Let $\mathcal{U}$ be an open cover of $X$.
We  say that a $G$-space $X$ is $G$-$\mathcal{U}$-dominated by a $G$-space $Y$ if there exist $G$-maps $f:X\to Y$ and $\varphi:Y\to X$ such that $\varphi f$ is $G$-$\mathcal{U}$-homotopic to the identity map $id_{X}$.

\medskip 
As in the non-equivariant case, $G$-ANR spaces have the $G$-homotopy extension property. 
Moreover, they also possess the equivariant extension property for $G$-$\mathcal{U}$-homotopies, as  presented in \cite{Ale2007}.

\begin{theorem}[{\cite[Theorem 5.1]{Ale2007}}]\label{Extension_homotopia}
Let $Y$ be a $G$-\textup{ANR} and let $\mathcal{U}$ be a $G$-cover of $Y$. 
Suppose $A$ is a closed invariant subset of a metrizable $G$-space $X$ and let $H_{t}:A\to Y$, $t\in [0,1]$, be a $G\text{-}\mathcal{U}$-homotopy. 
If $H_{0}$ can be extended to a $G$-map $f:X\to Y$, then there exists a $G\text{-}\mathcal{U}$-homotopy $\widetilde{H}_{t}:X\to Y$ such that $\widetilde{H}_{0}=f$ and $\widetilde{H}_{t}|_{A}=H_{t}$ for all $t\in [0,1]$.
\end{theorem}

Let $X$ and $Y$ be $G$-spaces, and let $\mathcal{U}$ be an open cover of $Y$. 
A $G$-map $f:X\to Y$ is called a \textit{$G$-$\mathcal{U}$-homotopy equivalence}, if there exists a $G$-map $\varphi:Y\to X$ such that $\varphi f\underset{G\text{-}f^{-1}(\mathcal{U})}{\simeq} id_{X}$ and $f\varphi\underset{G\text{-}\mathcal{U}}{\simeq} id_{Y}$. 
In this case,   $\varphi$ is called a \textit{$G$-$\mathcal{U}$-homotopy inverse} of $f$.

 Let $X$ and $Y$ be $G$-spaces. We say that a $G$-map $f:X\to Y$ is a \textit{fine $G$-homotopy equivalence}, if $f$ is a $G$-$\mathcal{U}$-homotopy equivalence for every open cover $\mathcal{U}$ of $Y$.   

\begin{theorem}[{\cite[Theorem 7]{Antonyan1988}}]\label{Teorema7}
Let $Y$ be a metrizable $G$-space. 
If for any open cover $\mathcal{U}$ of $Y$, there exists a $G$-\textup{ANR} space $X$ such that $Y$ is $G$-$\mathcal{U}$-homotopy dominated by $X$, then $Y$ is a $G$-\textup{ANR}.
\end{theorem}

It follows from this theorem that the image under a fine $G$-homotopy equivalence of a $G$-ANR is also a $G$-ANR.

\begin{corollary}\label{corTeorema7}
    Let $f:X\to Y$ be a fine $G$-homotopy equivalence. 
    If $X$ is a  $G$-\textup{ANR}, then $Y$ is a $G$-\textup{ANR}.
\end{corollary}
  
  \medskip

\section{$G$-ANR spaces  and fine $G$-homotopy equivalences}

We now extend the result presented in Corollary \ref{corTeorema7} by proving, in Theorem \ref{Gkozlowski}, the equivariant analogue of Kozlowski's Theorem (\cite[Theorem 6.7.5]{Sakai2013}). 
An important step toward this goal is the fact that $G$-ANRs have the following key refinements.

\medskip 
Let $Y$ be a $G$-space and let $\mathcal{U}$ be an open $G$-cover of $Y$. 
An open $G$-cover $\mathcal{V}$ of $Y$ is called an {\it $h$-$G$-refinement} of $\mathcal{U}$,
if $\mathcal{V}$ is a refinement of $\mathcal{U}$, and any two $\mathcal{V}$-close $G$-maps $f, \kappa : X \to Y$ defined on a metrizable $G$-space $X$ are $G$-$\mathcal{U}$-homotopic.

\begin{theorem}[{\cite[Theorem 4.2]{Ale2007}}] \label{V-cercanas}
    Every open $G$-cover of a $G$-\textup{ANR} space has an $h$-$G$-refinement.
\end{theorem}

 \begin{lemma}\label{remark:Kozlowski-2}
        Let $X\in G$-\textup{ANR}, $Y$ be a paracompact $G$-space and $f:X\to Y$ be a $G$-map. 
        Suppose that, for any open cover $\mathcal{U}$ of $Y$, there exists a $G$-map $\varphi:Y\to X$ such that 
        $\varphi f\underset{G\text{-}f^{-1}(\mathcal{U})}{\simeq}id_{X}$.
        Then, for every open $G$-cover $\mathcal{U}$ of $Y$ there exists an open $G$-cover $\mathcal{V}$ which is a refinement of $\mathcal{U}$, and $f^{-1}(\mathcal{V})$ is an $h$-$G$-refinement of $f^{-1}(\mathcal{U})$.
    \end{lemma}
    \begin{proof}
        Let $\mathcal{U}$ be an open $G$-cover of $Y$. 
        By Proposition \ref{gcubierta}, there exists an open $G$-cover $\mathcal{W}$ of $Y$ that is a star-refinement of $\mathcal{U}$. 
    
        Let $\varphi:Y\to X$ be a $G$-map such that
         $\varphi f\underset{G\text{-}f^{-1}(\mathcal{W})}{\simeq}id_{X}$.
      
        Since $X\in G$-\textup{ANR}, by Theorem \ref{V-cercanas}, there exists a $G$-cover $\mathcal{W}'$ of $X$ that is an $h$-$G$-refinement of $f^{-1}(\mathcal{W})$. 

        Let $\mathcal{V}=\{U\cap W\ |\ U\in \mathcal{U},\ W\in \varphi^{-1}(\mathcal{W}')\}.$
        
        Then $\mathcal{V}$ is a $G$-cover which  is a  refinement of both $\varphi^{-1}(\mathcal{W}')$ and $\mathcal{U}$.

        \medskip 
        Next, we show that $f^{-1}(\mathcal{V})$ is an $h$-$G$-refinement of $f^{-1}(\mathcal{U})$.
       
        It is clear that $f^{-1}(\mathcal{V})$ is a  refinement of both $(\varphi f)^{-1}(\mathcal{W}')$ and $f^{-1}(\mathcal{U})$. 
        
        Let $Z$ be a metrizable $G$-space, and let $\kappa$, $\kappa':Z\to X$ be two  $f^{-1}(\mathcal{V})$-close $G$-maps. Then $\kappa$ and $\kappa'$ are also $(\varphi f)^{-1}(\mathcal{W}')$-close, 
        which implies that $\varphi f\kappa$ and $\varphi f\kappa'$ are $\mathcal{W}'$-close.
        
        Since $\mathcal{W}'$ is an $h$-$G$-refinement of $f^{-1}(\mathcal{W})$, we have  
           $ \varphi f\kappa \underset{G\text{-}f^{-1}(\mathcal{W})}{\simeq} \varphi f\kappa'.$ 
       Moreover, since $\varphi f\underset{G\text{-}f^{-1}(\mathcal{W})}{\simeq}id_{X}$ 
        we get that  $ \varphi f\kappa \underset{G\text{-}f^{-1}(\mathcal{W})}{\simeq} \kappa$ and $ \varphi f\kappa' \underset{G\text{-}f^{-1}(\mathcal{W})}{\simeq} \kappa'$ which implies that $ \kappa \underset{G\text{-}St(f^{-1}(\mathcal{W}))}{\simeq} \kappa'$.
        
        Finally, since $\mathcal{W}$ is a star-refinement of $\mathcal{U}$, it follows that $St(f^{-1}(\mathcal{W}))$ is a refinement of $f^{-1}(\mathcal{U})$. 
        Thus $\kappa\underset{G\text{-}f^{-1}(\mathcal{U})}{\simeq} \kappa',$ as required.
    \end{proof}

\begin{theorem}\label{Gkozlowski}
  Let $f:X\to Y$ be a $G$-map where $X\in G$-\textup{ANR} and $Y$ is a metrizable $G$-space. Suppose that $f(X)$ is dense in $Y$ and that, for every open cover $\mathcal{U}$ of $Y$, there exists a $G$-map $\varphi:Y\to X$ such that $\varphi f\underset{G\text{-}f^{-1}(\mathcal{U})}{\simeq} id_{X}$. Then $f$ is a fine $G$-homotopy equivalence, and $Y$ is a $G$-\textup{ANR}.  
\end{theorem}

\begin{proof}  
    Let $\mathcal{U}$ be an open cover of $Y$. 
    Since $Y$ is a metrizable $G$-space, there exists a compatible invariant metric $\rho$ on $Y$ such that $\mathcal{B}=\{{B}_{\rho}(y,1)\ |\ y\in Y\}$ is a refinement of $\mathcal{U}$, where we denote by ${B}_{\rho}(y,r)$ the open ball of radius $r$ centered at the point $y$.
    
    We will  construct, by induction, a sequence of open $G$-covers $\{\mathcal{U}_{n}\}_{n\geq 0}$  of $Y$ satisfying the following conditions:
    \begin{enumerate}[(1)]
        \item $\mathcal{U}_{0}=\mathcal{U}$ and $\mathcal{U}_{n}$ is a star-refinement of $\mathcal{U}_{n-1}$,
        \item $f^{-1}(\mathcal{U}_{n})$ is an $h$-$G$-refinement of $f^{-1}(\mathcal{U}_{n-1})$,
        \item  $\text{mesh}_{\rho}(\mathcal{U}_{n})<2^{-n-1},$
    \end{enumerate}
     for all $n\in\mathbb{N}$.
       
       \smallskip
       
    Let  $\mathcal{B}_{1}=\{B_{\rho}(y, 2^{-4})\ |\ y\in Y\}$. 
    Clearly, $\text{mesh}_{\rho}(\mathcal{B}_{1})<2^{-2}$. 
     By  Proposition \ref{gcubierta}, there exists a $G$-cover $\mathcal{V}_{1}$ of $Y$ that is a star-refinement of $\mathcal{B}_{1}$. 
    Using Lemma \ref{remark:Kozlowski-2}, we construct a $G$-cover $\mathcal{U}_{1}$ which is a refinement of $\mathcal{V}_{1}$ and $f^{-1}(\mathcal{U}_{1})$ is an  $h$-$G$-refinement of $f^{-1}(\mathcal{V}_{1})$.    

    Now, assume that the $G$-covers  $\mathcal{U}_{1}, \mathcal{U}_{2},...,\mathcal{U}_{n-1}$, have been constructed. 
    Let  $\mathcal{B}_{n}=\{B_{\rho}(y,2^{-n-3})\ |\ y\in Y\}$,
    and let $\mathcal{W}_{n}$  be a common refinement of $\mathcal{B}_{n}$ and $\mathcal{U}_{n-1}$.
    Note that $\text{mesh}_{\rho}(B_{n})<2^{-n-1}$. 
    By Proposition \ref{gcubierta}, there exists an open $G$-cover $\mathcal{V}_{n}$  that is a star-refinement of $\mathcal{W}_{n}$. 
    Applying Lemma \ref{remark:Kozlowski-2}, we construct $\mathcal{U}_{n}$ which is a refinement of $\mathcal{V}_{n}$ and $f^{-1}(\mathcal{U}_{n})$ is an  $h$-$G$-refinement of $f^{-1}(\mathcal{V}_{n})$. 
    
    By  the hypothesis, for every $\mathcal{U}_{n}$, there exists a $G$-map  $\varphi_{n}:Y\to X$ such that
    \begin{equation}\tag{*}   \label{eq1}
    \varphi_{n}f\underset{G\text{-}f^{-1}(\mathcal{U}_{n+2})}{\simeq} id_{X}.
    \end{equation}
    Using  Lemma \ref{remmark:Kozlowski-1},  we see that $f\varphi_{n}$ is $St(\mathcal{U}_{n+2})$-close to $id_{Y}$. 
    Since $\mathcal{U}_{n+2}$ is a star-refinement of $\mathcal{U}_{n+1}$, it follows that $f\varphi_{n}$ and $id_{Y}$ are  $\mathcal{U}_{n+1}$-close. 
    Additionally, $f\varphi_{n}$ and $f\varphi_{n+1}$ are $\mathcal{U}_{n}$-close, 
    so $\varphi_{n}$ and $\varphi_{n+1}$ are $f^{-1}(\mathcal{U}_{n})$-close. 
    
    Thus, for every $n\in\mathbb{N}$, there exists a $G$-$f^{-1}(\mathcal{U}_{n-1})$-homotopy \linebreak $F_{n}:Y\times [0,1]\to X$ such that $$F_{n}(y,0)=\varphi_{n}(y)\qquad \text{and}\qquad F_{n}(y,1)=\varphi_{n+1}(y),\qquad \forall y\in Y.$$

    Define $H:Y\times [0,1]\to Y$ by
    $$
     H(y,t)=\begin{cases}
                f(F_{n}(y,2-2^{n}t)),& \text{if}\ \ 2^{-n}\leq t\leq 2^{-n+1}\\
                \qquad \quad y, & \text{if}\ \ t=0.
            \end{cases}
    $$
    
    Let us verify that $H$ is continuous on $Y\times \{0\}$. 
    
    Fix $t>0$. Then, there exists $n\in \mathbb{N}$ such that $2^{-n}\leq t\leq 2^{-n+1}$. 
    Since $\text{mesh}_{\rho}(\mathcal{U}_{n})<2^{-n}$, for every $y\in Y$, we have:
    \begin{align*}
        \rho(H(y,t),y) = & \rho(fF_{n}(y,2-2^{n}t),y)\\
        \leq &\rho(fF_{n}(y,2-2^{n}t),f\varphi_{n}(y))+\rho(f\varphi_{n}(y),y)\\
        < & 2^{-n}+2^{-n-2}<2^{-n}2\leq 2t.
    \end{align*}
    This establishes the continuity of $H$ at $Y\times \{0\}$.
    Moreover, since $f$ and $F_n$ are equivariant, $H$ inherits this property as well.

    Besides, since
    \begin{align*}
        \text{diam}_\rho(H(\{y\}\times [0,1])) &= \text{diam}_{\rho}(H(\{y\}\times (0,1])\\
        &= \text{diam}_{\rho}\left(\bigcup_{n\in\mathbb{N}}fF_{n}(\{y\}\times (0,1])\right)\\
        &\leq \sum_{n\in\mathbb{N}} \text{mesh}_{\rho}(\mathcal{U}_{n-1})< \sum_{n\in\mathbb{N}} 2^{-n}=1,
    \end{align*}
    we conclude that $H$ is a $G$-$\mathcal{U}$-homotopy. 
    
    Furthermore, since $H(y,0)=y$ and $H(y,1)=f\varphi_{1}(y)$ for all $y\in Y$, it follows that 
    $f\varphi_{1}\underset{G\text{-}\mathcal{U}}{\simeq}id_{Y}$. 

    It follows from \eqref{eq1}  that $\varphi_{1}f$ is $G\text{-}f^{-1}(\mathcal{U}_{3})$-homotopic to $id_{X}$.
    
    Hence, $\varphi_{1}$ is a $G$-$\mathcal{U}$-homotopy inverse of $f$, and therefore, $f$ is a fine $G$-$\mathcal{U}$-homotopy equivalence.

    Finally, applying Corollary \ref{corTeorema7}, we conclude that $Y\in G$-ANR, completing the proof. 
\end{proof}

\medskip 

\section{$G$-homotopy dense subsets and $G$-ANR\ spaces}

$G$-homotopy dense subsets are of particular interest because 
they inherit the property of being a $G$-ANR. Moreover, if a $G$-space contains a $G$-homotopy dense subset that is a $G$-ANR, then the $G$-space itself is also a $G$-ANR. 
This relationship will be proved in Propositions \ref{GdensodeGANR} and \ref{recip}.

\begin{definition}
Let $A$ be an invariant subset of a $G$-space $X$. We will say that $A$ is \textit{$G$-homotopy dense} in $X$, if there exists a $G$-homotopy $F:X\times I\to X$, where $F_{0}=id_{X}$ and $F_{t}(X)\subset A$ for every $t\in(0,1]$.  
\end{definition}

\begin{proposition}\label{GdensodeGANR}
Let  $A$ be a  $G$-homotopy dense subset of a $G$-space $X$. 
If $X$ is a $G$-\textup{ANR} then $A$ is also a $G$-\textup{ANR}.
\end{proposition}

\begin{proof}
Since $A$ is $G$-homotopy dense, there exists a $G$-homotopy $F:X\times I\to X$ such that $F_{0}(x)=x$ and $F_{t}(x)\in A$ for all $x\in X$ and $t\in(0,1]$.

Let $Y$ be a metrizable $G$-space, $B$ an invariant closed subset of $Y$ and $f:B\to A$ a $G$-map. 
Since $X$ is a $G$-ANR, there exists an invariant neighborhood $U$ of $B$ in $Y$ and a $G$-map $\overline{f}:U\to X$ such that $\overline{f}|_B=f$.

Let $d$ be an invariant metric on $Y$ such that $\operatorname{diam}\,Y<1$. 
Define $\widetilde{f}:U\to A$  by  
$$\widetilde{f}(x)=F_{d(x,B)}\big(\overline{f}(x)\big)$$ for all $x\in U.$ 

First, we verify that $\widetilde{f}(U)\subset A$.
Let $x\in U$.  
If $x\in B$ then  $d(x,B)=0$, so $\widetilde{f}(x)=F_{0}(\overline{f}(x))=\overline{f}(x)=f(x)\in A$. 

If $x\notin B$ then  $d(x,B)>0$ and $\widetilde{f}(x)=F_{d(x,B)}(\overline{f}(x))\in A$.
Therefore, $\widetilde{f}(U)\subset A$.

Additionally, $\widetilde{f}$ is continuous as it is a composition of continuous maps. 

To verify that $\widetilde{f}$ is an equivariant map, we  use the fact that 
$F$ and $\overline{f}$ are equivariant maps and  $d$ is an invariant metric.

Let  $g\in G$ and $x\in U$.
Then, 
   \begin{align*}
    \widetilde{f}(gx)&=F_{d(gx,B)}\big(\overline{f}(gx)\big)
    =F_{d(gx,gB)}\big(g\overline{f}(x)\big) 
    =F_{d(x,B)}\big(g\overline{f}(x)\big) \\
    &=gF_{d(x,B)}\big(\overline{f}(x)\big) 
    =g\widetilde{f}(x).
\end{align*}
    
Therefore, $\widetilde{f}$ is the desired $G$-extension of $f$.
\end{proof}

\medskip 
As we will see in Theorem \ref{gdenso-homofina}, $G$-homotopy dense subsets can be characterized by fine $G$-homotopy equivalences. 
For this purpose, we first establish the following lemma.

\begin{lemma}\label{g-funcionsemicontinua}
Let $X$ be a paracompact $G$-space and let $s,h:X\to \mathbb{R}$ be invariant maps such that $s$ is upper semi-continuous, $h$ is lower semi-continuous, and $s(x)<h(x)$ for every $x\in X$. 
Then there exists an invariant continuous map $f:X\to \mathbb{R}$ such that $s(x)<f(x)<h(x)$ for every $x\in X$.
\end{lemma}
\begin{proof}
  By the compactness of $G$, the orbit map $\pi :X\to X/G$, $x\mapsto G(x)$, is a perfect map (see \cite[Chapter I, Theorem 3.1]{Bredon72}).
  This implies that $X/G$ is a paracompact space. 
  The invariant maps $h$ and $s$  induce  continuous  maps $\tilde{s},\tilde{h}:X/G\to \mathbb{R}$ defined by  
$$\tilde{s}(G(x))=s(x) \ \text{ and } \  \tilde{h}(G(x))=h(x)\quad \text{for every}\ \ G(x)\in X/G.$$

Observe that  $\tilde{s}$ is an upper semi-continuous map because  $\pi$ is an open map, $\tilde{s}^{-1}(-\infty,t)=  \pi(s^{-1}(-\infty,t))$ 
 and  $s^{-1}(-\infty, t)$ is open in $X$ due to the upper semi-continuity of $s$.  
Similarly, $\tilde{h}$ is a lower semi-continuous map. 

 By the Kat\v{e}tov--Tong Insertion Theorem (\cite[Theorem 1]{Katetov}), there exists a continuous map 
 $\tilde{f}:X/G\to \mathbb{R}$ such that, for every $x\in X$,
 $$\tilde{s}(G(x))<\tilde{f}(G(x))<\tilde{h}(G(x)).$$  
 
 Define $f=\tilde{f}\circ\pi:X\to \mathbb{R}$. 
 The map $f$ is  continuous and invariant, and
 for every $x\in X$ one has
 $$s(x)=\tilde{s}(G(x))<\tilde{f}(G(x))=\tilde{f}(\pi(x))=f(x)$$
 and
 $$f(x)=\tilde{f}(\pi(x))=\tilde{f}(G(x))<\tilde{h}(G(x))=h(x).$$
 Thus, $f$ satisfies the desired properties.
\end{proof}

\medskip 
\begin{theorem}\label{gdenso-homofina}
 Let $X$ be a metrizable $G$-space and $A$ an invariant subset of $X$. 
 Then $A$ is $G$-homotopy dense in $X$ if and only if the inclusion map $i: A\hookrightarrow X$ is a fine $G$-homotopy equivalence.   
\end{theorem}
\begin{proof}
$(\Rightarrow)$ 
Suppose that $A$ is $G$-homotopy dense in $X$. 
Then there exists a $G$-homotopy $F:X\times I\to X$ such that $F_{0}(x)=x$ and $F_{t}(x)\in A$ for every $x\in X$ and $t\in(0,1]$. 

Let $\mathcal{U}$ be an open cover of $X$. By Proposition \ref{gcubierta}, we  can assume that $\mathcal{U}$ is a $G$-cover. 

Let us define $\gamma:X\to (0,1]$ as  
$$\gamma(x)=\sup \{t\in I \mid \exists  U\in \mathcal{U} \ \text{such that } \  F(\{x\}\times [0, t]) \subset U\}$$
for  every $ x\in X$. 

The function $\gamma$ is well defined due to the continuity of $F$.
Moreover, since $\mathcal{U}$ is a $G$-cover, it follows that $\gamma(gx)=\gamma(x)$ for all $x\in X$ and $g \in G$.

Next, we will show  that $\gamma$ is lower semi-continuous. 

Let $r>0$ and $x\in \gamma^{-1}(r,\infty)$. 
Then $\gamma(x)>r$.
 There exists $t\in I$ such that $r<t<\gamma(x)$ and  
 $F(\{x\}\times [0, t]) \subset U$ for some $ U\in \mathcal{U}$. 
 
 By the continuity of $F$ and the compactness of $[0,t]$, there exists an open neighborhood $V$ of $x$ in $X$ such that $F(V\times [0,t])\subset U$.
 
 Consequently, $\gamma(y)\geq t>r$ for all $y\in V$  
 implying that $V\subset \gamma^{-1}(r,\infty)$. Thus,  $\gamma^{-1}(r,\infty)$ is open in $X$, 
 proving that $\gamma$ is lower semi-continuous.
 
Using Lemma \ref{g-funcionsemicontinua}, we can find an invariant continuous map 
$\alpha:X\to (0,1]$ such that $\alpha(x)<\gamma(x)$ for all $x\in X$. 

Now define 
$$F^{\alpha}:X\times [0, 1]\to X\ \ \text{ by} \ \ 
F^{\alpha}(x, t) = F \bigl( x, t \alpha(x) \bigr)$$
 for every $x\in X$ and $t\in I$.

 The map $F^{\alpha}$ is continuous as it is the composition of continuous maps, 
 and it is equivariant since both $F$ and $\alpha$ are $G$-maps.

  We now verify that $F^{\alpha}$ is  a $G$-$\mathcal{U}$-homotopy. 
Since $\alpha(x)<\gamma(x)$,  for every $x\in X$, there exists $t_{0}\in [0, 1]$ such that $\alpha(x)<t_0<\gamma(x)$ and $F(\{x\}\times [0, t_0])\subset U$ for some $U\in \mathcal{U}$. 
 Since for every $t\in I$,  $t\alpha(x)\leq\alpha(x)<t_{0}$, we have  
 $F^{\alpha}(x, t)=F\big(x, t\alpha(x)\big)\in  F\big(\{x\}\times [0, t_0]\big)\subset  U$, as required.
 
Furthermore, since $\alpha(x)>0$, then $F^{\alpha}(x, 1)=F\big(x, \alpha(x)\big)\in A$ for every $x\in X$.  

Define $h:X\to A$  by $h(x)=F^{\alpha}(x, 1)$.
It follows that 
$i\circ h\underset{G\text{-}\mathcal{U}}{\simeq} id_{X}$ and $h\circ i\underset{G\text{-}i^{-1}(\mathcal{U})}{\simeq} id_{A}$, proving  that $i:A\hookrightarrow X$ is a fine $G$-homotopy equivalence.

$(\Leftarrow)$ 
Suppose that the inclusion map $i:A\hookrightarrow X$ is a fine $G$-homotopy equivalence. 
Let $d$ be a compatible metric for $X$, and  
define, for every $n\in\mathbb{N}$, the open cover $\mathcal{U}_{n}=\{B_{d}(x,3^{-n+1}2^{-1})\ |\ x\in X\}$ which clearly satisfies $\operatorname{mesh}\ \mathcal{U}_{n}< 3^{-n}$.

By hypothesis, for every $n\in \mathbb{N}$, there exist $G$-maps $f_{n}:X\to A$, $G$-$\mathcal{U}_{n}$-homotopies $F_{n}:X\times I\to X$ and $G$-$i^{-1}(\mathcal{U}_{n})$-homotopies $T_{n}:A\times I\to A$ satisfying:
$$F_{n}(x,0)=x\ \text{and}\ F_{n}(x,1)=f_{n}(x)\ \text{for all}\ x\in X$$
and $$T_{n}(a,0)=a\ \text{and}\ T_{n}(a,1)=f_{n}(a)\ \text{for all}\ a\in A.$$

Define $F:X\times I\to X$ as follows:
\medskip
$$F(x,t)=\begin{cases}
f_{n}F_{n+1}(x,3-3^{n}t), & \text{if}\ \ 3^{-n}2\leq t\leq 3^{-n+1},\\
T_{n}\big(f_{n+1}(x),3^{n}t-1\big), & \text{if}\ \ 3^{-n}\leq t\leq 3^{-n}2,\\
x, &  \text{if}\ \ t=0.
\end{cases}$$

Let us verify the continuity of $F$ at $X \times \{0\}$. 
Let $t > 0$. Then there exists $n \geq 1$ such that $3^{-n} \leq t \leq 3^{-n+1}$. 

We will consider two cases: $3^{-n}2 \leq t$ \ and \  $t \leq 3^{-n}2$.

Case 1. If $3^{-n}2\leq t$ then for every $x\in X$ we have 
\begin{align*}
d\big(&F(x, t),x\big)= d\big(f_{n}F_{n+1}(x,3-3^{n}t),x\big)\\
& \leq d\big(f_{n}F_{n+1}(x,3-3^{n}t),F_{n+1}(x,3-3^{n}t)\big)+ d\big(F_{n+1}(x,3-3^{n}t),x\big)\\
&< 3^{-n}+3^{-n+1}<3^{-n}2\leq 2t.
\end{align*}

Case 2. If $t\leq 3^{-n}2$ a similar argument shows that $d(F(x, t),x)<2t$ for every $x\in X$. 

\medskip
Now, let $\varepsilon>0$ and $x_{0}\in X$ be fixed. Define $\delta = \frac{\varepsilon}{2}$ and let  $x\in X$ satisfying $d(x,x_{0})<\delta$. 
Then for every $0<t<\frac{\varepsilon}{2}$ we have 
$$d\big(F(x,t),x_{0}\big)\leq d\big(F(x,t),x\big)+d(x,x_{0})<\frac{\varepsilon}{2}+\frac{\varepsilon}{2}=\varepsilon.$$
If $t=0$ then   $d(F(x,0),x_{0})=d(x,x_{0})<\delta <\varepsilon$. 
Therefore, $$F\big(B_{d}(x_{0},\delta)\times [0,\varepsilon/2)\big)\subset B_{d}(x_{0},\varepsilon).$$
This proves the continuity of $F$ at $X\times \{0\}$. 

\medskip 
It is evident from the definition of $F$ that it is an equivariant map and that $F(X\times (0,1])\subset A$.
Hence, $A$ is $G$-homotopy dense in $X$.
\end{proof}

\medskip 
As a direct consequence of Theorem \ref{gdenso-homofina} and Corollary \ref{corTeorema7}, we get the converse of Proposition \ref{GdensodeGANR}.

\begin{proposition} \label{recip}
Let $X$ be a metrizable $G$-space and $A\subset X$  a $G$-homotopy dense subset. If $A$ is a $G$-\textup{ANR} then $X$ is also  a $G$-\textup{ANR}.
\end{proposition}

Additionally, Theorem \ref{gdenso-homofina} allows us to derive the following immediate consequence of Theorem \ref{Gkozlowski}.

\begin{proposition}
    Let $X$ be a metrizable $G$-space and  $A$  an invariant subset of $X$. 
    Assume that $A$ is a $G$-\textup{ANR} which is dense in $X$ and that for every open cover $\mathcal{U}$ of $X$ there exists a $G$-map $\varphi:X\to A$ such that $\varphi|_{A}\underset{G\text{-}\mathcal{U}}{\simeq} id_{A}$. 
    Then $A$ is $G$-homotopy dense in $X$ and $X$ is a $G$-\textup{ANR}.
\end{proposition}

\medskip

\section{Equivariant inverse sequences and $G$-ANR spaces}

An {\it inverse sequence} in the category  of $G$-spaces and $G$-maps,
denoted by $\underline{X}=\{X_{i},p^{j}_{i}\}$, consists of  $G$-spaces $X_i$ for each $i\in \mathbb{N}$
and  $G$-maps $p^{j}_{i}:X_j\to X_i$ for all $i,j\in\mathbb{N}$ with  $i\leq j$. 
These  $G$-maps are called bonding maps and satisfy the properties $p^{j}_{j}=1_{X_{j}}:X_j\to X_j$, and 
$p^{j}_{i}\circ p^{k}_{j}=p^{k}_{i}$ for $i\leq j \leq k$.

Given an inverse sequence $\underline{X}=\{X_{i},p^{j}_{i}\}$ and a $G$-space $Y$, 
a family of $G$-maps $\{q_i:Y\to X_i\}_{i\in\mathbb{N}}$, known as the projections,
is called a {\it cone over} $\underline{X}$ if it satisfies $p^j_i\circ q_j=q_i$ for all $i\leq j$.

The {\it inverse limit} of $\underline{X}=\{X_{i},p^{j}_{i}\}$, denoted by $X=\mathop{lim}\limits_{\longleftarrow}\underline{X}$, is a $G$-space $X$ together with a cone over $\underline{X}$, 
$\{p_i:X\to X_i\}_{i\in\mathbb{N}}$ that
 satisfies the following universal property: \\
if $Y$ is a $G$-space and $\{q_i:Y\to X_i\}_{i\in\mathbb{N}}$ is a cone over $\underline{X}$, 
then there exists a unique $G$-map $f:Y\to X$ such that
$p_i\circ f=q_i$, for all $i\in\mathbb{N}$.
  
The inverse limit can be explicitly described  as 
$$\mathop{lim}\limits_{\longleftarrow}\underline{X}=\Big\{x\in\prod_{i\in\mathbb{N}}X_{i} \mid p_{i}(x)=p^{j}_{i}p_{j}(x), \, i\leq j\Big\},$$
where $p_j$ is the projection onto the $j$-th factor.

This space is equipped with the induced topology and  the diagonal action of $G$ making it a $G$-space.

 \bigskip
Now  we are prepared to present the equivariant extension of an important result of D.W. Curtis \cite[Theorem 3.2]{Curtis71}. This result establishes sufficient conditions under which the inverse limit of $G$-ANR spaces is a $G$-ANR.

\begin{theorem} \label{GCurtis}
Let $\underline{X}=\{X_{i},p^j_{i}\}$ be an inverse sequence of completely metrizable $G$-spaces and let $X=\mathop{lim}\limits_{\longleftarrow}\underline{X}$ be its inverse limit. 
Assume that each $X_{i}$ is a $G$-\textup{ANR} and every bonding map $p^{i+1}_{i}:X_{i+1}\to X_{i}$ is a fine $G$-homotopy equivalence. Then $X$ is a $G$-\textup{ANR}. Moreover, if every $X_{i}$ is a $G$-\textup{AR}, then $X$ is also a $G$-\textup{AR}.
\end{theorem}
\begin{proof}
    Let $A$ be an invariant closed subset of a metrizable $G$-space $Y$, and let $f:A\to X$ be a $G$-map.   
   
    Let $\{p_i:X\to X_i\}_{i\in\mathbb{N}}$ be the inverse limit projections. 
    Since $X_{1}$ is a $G$-ANR, the $G$-map $p_{1}f:A\to X_{1}$ admits an equivariant extension $\widetilde{f}_{1}:V\to X_{1}$ defined on some invariant neighborhood $V$ of $A$ in $Y$.
    
    We will construct,  for every $n\in \mathbb{N}$,  a $G$-map $\widetilde{f}_{n}:V\to X_{n}$ such that $\widetilde{f}_{n}|_{A}=p_{n}f$. 
    
    This will be done by induction on $n$.
    Assume $\widetilde{f}_{n}$ is given. 

 For each $i\in \mathbb{N}$, let $d_{i}$ be a complete invariant compatible metric  on $X_{i}$. 
    For every $i\in\{1,\ldots, n\}$, the collection $\mathcal{V}_{i}=\{B_{d_{i}}(x,2^{-n-3})\ |\ x\in X_{i}\}$ is an open  $G$-cover of $X_{i}$ since $d_{i}$ is an invariant metric.
    Moreover, $\text{mesh}_{d_{i}} \mathcal{V}_{i}<2^{-n-1}$.
    
    Let $\mathcal{U}_{n}$ be a $G$-cover of $X_n$ that is a common refinement of $(p^n_i)^{-1}(\mathcal{V}_{i})$  for every $i\in\{1,\ldots, n\}$. 
    Thus,
    \begin{equation*}
    \operatorname{mesh}_{d_{i}} p^n_i(\mathcal{U}_{n})<2^{-n-1}.
    \end{equation*}
    
 Since $p^{n+1}_{n}:X_{n+1}\to X_{n}$ is a fine $G$-homotopy equivalence,  there exists a $G$-map $\psi_{n}:X_{n}\to X_{n+1}$ such that
    \begin{equation*}
        \psi_{n}p^{n+1}_{n}\underset{G\text{-}(p^{n+1}_{n})^{-1}(\mathcal{U}_{n})}{\simeq} id_{X_{n+1}}\quad \text{and}\quad p^{n+1}_{n}\psi_{n}\underset{G\text{-}\mathcal{U}_{n}}{\simeq}id_{X_{n}}.
    \end{equation*}
    Then,
    $$\psi_{n}p_{n}f=\psi_{n}p^{n+1}_{n}p_{n+1}f\underset{G\text{-}(p^{n+1}_n)^{-1}(\mathcal{U}_{n})}{\simeq}id_{X_{n+1}}p_{n+1}f=p_{n+1}f,$$
     which implies the existence of a $G$-$(p^{n+1}_n)^{-1}(\mathcal{U}_{n})$-homotopy $H_{t}:A\to X_{n+1}$ such that $H_{0}=\psi_{n}p_nf$ and $H_{1}=p_{n+1}f$. 
     
     Since $X_{n+1}$ is a $G$-ANR and $\psi_n\widetilde{f}_n:V\to X_{n+1}$ is a $G$-extension of $H_{0}=\psi_{n}p_nf$, 
     applying Theorem \ref{Extension_homotopia}, 
    we get a $G$-$(p^{n+1}_n)^{-1}(\mathcal{U}_{n})$-homotopy $\widetilde{H}_{t}:V\to X_{n+1}$ such that $\widetilde{H}_0=\psi_n\widetilde{f}_n$ and $\widetilde{H}_t|_A=H_t$ for all $t\in I$.
    Define $\widetilde{f}_{n+1}=\widetilde{H}_1:V\to X_{n+1}$. 
    This map satisfies $\widetilde{f}_{n+1}|_A=p_{n+1}f$.

    \medskip
    Moreover, since $\widetilde{f}_{n+1}$ and $\psi_n\widetilde{f}_n$ are 
    $G$-$(p^{n+1}_n)^{-1}(\mathcal{U}_{n})$-homotopic,  there exists a $U\in \mathcal{U}_n$ such that, for every $v\in V$,
    $$\widetilde{f}_{n+1}(v),\, \psi_n\widetilde{f}_n(v)\in (p^{n+1}_n)^{-1}(U).$$
    Consequently,
    $$p^{n+1}_{n}\widetilde{f}_{n+1}(v),\ p^{n+1}_{n}\psi_{n}\widetilde{f}_{n}(v) \in U,$$
     and therefore,
     $$p^n_ip^{n+1}_{n}\widetilde{f}_{n+1}(v),\ p^n_ip^{n+1}_{n}\psi_{n}\widetilde{f}_{n}(v) \in p^n_i(U),$$
     which implies that
      $$p^{n+1}_{i}\widetilde{f}_{n+1}(v),\; p^{n+1}_{i}\psi_{n}\widetilde{f}_{n}(v) \in p^{n}_{i}(U),$$
    for every $i\in\{1,\ldots, n\}$.

\medskip
     On the other hand, since $p^{n+1}_{n}\psi_{n}\underset{G\text{-}\mathcal{U}_{n}}{\simeq}id_{X_{n}}$, there exists $U'\in \mathcal{U}_{n}$ such that for every $x\in X_n$, 
     $$p^{n+1}_{n}\psi_n(x),\; x\in U'.$$
     Thus, for every $v\in V$,
     $$p^{n+1}_{n}\psi_{n}\widetilde{f}_{n}(v),\ \widetilde{f}_n(v)\in U',$$
     and therefore, 
     $$p^n_ip^{n+1}_{n}\psi_{n}\widetilde{f}_{n}(v),\ p^n_i\widetilde{f}_n(v)\in p^n_i(U'),$$
     which implies that
     $$p^{n+1}_{i}\psi_{n}\widetilde{f}_{n}(v),\ p^n_i\widetilde{f}_n(v)\in p^n_i(U')$$
    for every $i\in\{1,\ldots, n\}$.

    Since $\operatorname{mesh}_{d_{i}} p^n_i(\mathcal{U}_{n})<2^{-n-1}$, 
    by the triangle inequality, we get that  \medskip
    $$
    d_{i}(p^{n+1}_{i}\widetilde{f}_{n+1}(v),p^{n}_{i}\widetilde{f}_{n}(v))<2^{-n}
    $$
    for every $v\in V$ and $i\in\{1,\ldots,n\}$.
\medskip

This yields that, for a fixed $i\in\{1,\ldots,n\}$,
the functional sequence $$\big\{p^{n}_{i}\widetilde{f}_{n}(v)\big\}_{n=1}^\infty, \quad v\in V$$
is a uniform (with respect to $v\in V$) Cauchy sequence. 

Therefore, due to the completeness of the metric $d_i$, the sequence $\big\{p^{n}_{i}\widetilde{f}_{n}(v)\big\}_{n=1}^\infty$  converges uniformly, with respect to $v\in V$, to a continuous map $q_i(v)$, $v\in V$.

    \medskip 
  In other words, we can define a continuous  map $$q_{i}=\lim_{n\to \infty}p^{n}_{i}\widetilde{f}_{n}:V\to X_{i}.$$
  The equivariance of the maps $p^{n}_{i}$ and $\widetilde{f}_{n}$ easily imply the
equivariance of $q_i$.

   Observe that, for every $i\in\mathbb{N}$, 
    \begin{align*}
        p^{i+1}_{i}q_{i+1}&=p^{i+1}_{i}\left(\lim_{n\to \infty}p^{n}_{i+1}\widetilde{f}_{n} \right)\\
        &= \lim_{n\to \infty}p^{i+1}_{i}p^{n}_{i+1}\widetilde{f}_{n} \\
        &= \lim_{n\to \infty} p^n_i\widetilde{f}_n\\
        &= q_{i}.
    \end{align*}

    Thus, by the universal property of the inverse limit, there exists a unique $G$-map $h:V\to X$ such that $p_{i}h=q_{i}$ for all $i\in \mathbb{N}$. 

    Finally,  since $p^{n}_{i}\widetilde{f}_{n}|_{A}=p^{n}_{i}p_{n}f=p_{i}f$,
   it follows that $q_{i}|_{A}=p_{i}f$. 
    Hence, $p_{n}h|_{A}=q_{n}|_{A}=p_{n}f$ for every $n\in \mathbb{N}$,  implying that $h|_{A}=f$. 
    Therefore, $h$ is the desired $G$-extension of $f$.

    If each $X_i$ is a $G$-AR,  we can take $V=Y$ and $f$ will have an equivariant extension defined over whole $Y$.
\end{proof}

\end{document}